\documentclass{amsart}
\usepackage{amssymb}
\usepackage{amsfonts}
\usepackage{cite}

\newtheorem{innercustomthm}{Theorem}
\newenvironment{customthm}[1]
{\renewcommand\theinnercustomthm{#1}\innercustomthm}
{\endinnercustomthm}
\newtheorem{theorem}{Theorem}
\theoremstyle{plain}

\newtheorem{corollary}{Corollary}

\newtheorem{example}{Example}

\newtheorem{lemma}{Lemma}

\newtheorem{proposition}{Proposition}
\newtheorem{remark}{Remark}

\numberwithin{equation}{section}

\begin{document}
\title{$2-$Killing Vector Fields on Warped Product Manifolds}
\author{S. Shenawy}
\address{\newline Basic Science Department,\newline Modern Academy for Engineering and Technology,\newline Maadi, Egypt,}

\email{drssshenawy@eng.modern-academy.edu.eg }
\author{B. Unal}
\address{Department of Mathematics, Bilkent University, Bilkent, 06800
Ankara, Turkey}

\email{bulentunal@mail.com }
\subjclass[2010]{53C21, 53C25, 53C50, 53C80}
\keywords{Warped product, $2-$Killing vector field, parallel vector fields,
standard static space-times.}

\begin{abstract}
The present article provides a study of $2-$Killing vector fields on
warped product manifolds as well as characterization of this
structure on standard static and generalized Robertson-Walker
space-times. Some conditions for a $2-$Killing vector field on
a warped product manifold to be parallel are obtained. Moreover,
some results on the curvature of a warped product manifolds in
terms of $2-$Killing vector fields are derived. Finally, we apply
our results to describe $2-$Killing vector fields of some well
known warped product space-time models.
\end{abstract}

\maketitle

\section{Introduction}

Killing vector fields have a well-known geometrical and physical
interpretations and have been studied on Riemannian and pseudo-Riemannian
manifolds for a long time. The number of independent Killing vector fields
measures the degree of symmetry of a Riemannian manifold. Thus, the problems
of existence and characterization of Killing vector fields are important and
are widely discussed by both mathematicians and physicists \cite%
{Al-Solamy:2011,Apostolopoulos:2005,Berestovskii:2008,Besse2008,Ivancevic:2007,Unal:2012,Duggala:2005,Nomizo:1960,Semmelmann:2003,Yorozu:1982}

Generalization of Killing vector fields has a long history in mathematics
for different scales and purposes\cite%
{Blair:1971,Deshmukh:2012,Deshmukh:2014,Kuhnel:1997}. In \cite{Operea:2008},
the concept of $2-$Killing vector fields, as a new generalization of Killing
vector fields, was first introduced and studied on Riemannaian manifolds.
The relations between $2-$Killing vector fields, curvature and monotone vector
fields are obtained. Finally, a characterization of $2-$Killing vector field
on $\mathbb{R}^{n}$ is derived.

At this point, we want to emphasize that the concept of monotone vector fields
introduced by S. Z. N\'{e}meth (see \cite{SZ3,SZ1,SZ2,SZ4}) and since then
they have been studied as a research topic in the area of (nonlinear)
analysis on Riemannian manifolds (see also \cite{Barani} as an additional
reference to the above list). As noted above, the connections
between monotone vector fields and $2-$Killing vector fields have been
established. In addition to that, by using space-like hypersurfaces of
a Lorentzian manifold (see \cite{Neto2015}), these topics have been received
attention in Lorentzian geometry as well. Thus the notion of $2-$Killing vector
fields is important in different branches of mathematics from (nonlinear)
analysis on Riemannian manifolds to Lorentzian geometry.

As far as we know, the concept of $2-$Killing vector fields has been
studied neither on warped products nor on space-time models up to this
paper in which we intent to fill this gap in the literature by providing
a complete study of $2-$Killing vector fields on such spaces. In this way,
all the results related to $2-$Killing vector fields and thus monotone vector
fields can be easily extended to a special class of manifolds, namely, warped
product manifolds.

We organize the paper as follows. In Section 2, we state well known connection
related formulas of warped product manifolds and Killing vector fields.
Thus some of the proofs are omitted. In Section 3, as the core of the paper,
the relation between $2-$Killing vector fields on a warped product manifold
and $2-$Killing vector fields on the fiber and base manifolds is discussed.
Here, we now state our main results; the following theorem represents an
important and helpful identity.

\begin{customthm}{\ref{TH 1}}
Let $\zeta =\left( \zeta _{1}, \zeta _{2}\right) \in \mathfrak{X}\left(
M_{1}\times _{f}M_{2}\right) $ be a vector field on a warped product
manifold of the form $M_{1}\times _{f}M_{2}$. Then%
\begin{eqnarray*}
\left( \mathcal{L}_{\zeta }\mathcal{L}_{\zeta }g\right) \left( X,Y\right)
&=&\left( \mathcal{L}_{\zeta _{1}}^{1}\mathcal{L}_{\zeta
_{1}}^{1}g_{1}\right) \left( X_{1},Y_{1}\right) +f^{2}\left( \mathcal{L}%
_{\zeta _{2}}^{2}\mathcal{L}_{\zeta _{2}}^{2}g_{2}\right) \left(
X_{2},Y_{2}\right) \\
&&+4f\zeta _{1}\left( f\right) \left( \mathcal{L}_{\zeta
_{2}}^{2}g_{2}\right) \left( X_{2},Y_{2}\right) +2f\zeta _{1}\left( \zeta
_{1}\left( f\right) \right) g_{2}\left( X_{2},Y_{2}\right) \\
&&+2\zeta _{1}\left( f\right) \zeta _{1}\left( f\right) g_{2}\left(
X_{2},Y_{2}\right)
\end{eqnarray*}%
for any vector fields $X,Y\in \mathfrak{X}\left( M_{1}\times
_{f}M_{2}\right) $.
\end{customthm}

The proof of this result contains long computations that have been done
using previous results on warped product manifolds (see appendix A). As an
immediate consequence, the relation between $2-$Killing vector fields on a
warped product manifold and those on product factors is given.

Some conditions for a $2-$Killing vector field to be parallel vector field is
considered in the following theorem.

\begin{customthm}{\ref{TH2}}
Let $\zeta \in \mathfrak{X}\left( M_{1}\times _{f}M_{2}\right) $ be a vector
field on a warped product manifold of the form $M_{1}\times _{f}M_{2}$. Then

\begin{enumerate}
\item $\zeta =\zeta _{1}+\zeta _{2}$ is parallel if $\zeta _{i}$ is a $2-$%
Killing vector field, $Ric^{i}\left( \zeta _{i},\zeta _{i}\right) \leq 0,$ $%
i=1,2$ and $f$ is constant.

\item $\zeta =\zeta _{1}$ is parallel if $\zeta _{1}$ is a $2-$Killing
vector field, $Ric^{1}\left( \zeta _{1},\zeta _{1}\right) \leq 0,$ and $%
\zeta_{1}\left( f\right) =0$.

\item $\zeta =\zeta _{2}$ is parallel if $\zeta _{2}$ is a $2-$Killing
vector field, $Ric^{2}\left( \zeta _{2},\zeta _{2}\right) \leq 0,$ and $f$
is constant.
\end{enumerate}
\end{customthm}

The preceding theorem also provides some results on the curvature of a
warped product manifold in terms of $2-$Killing vector fields.

\begin{customthm}{\ref{TH3}}
Suppose that $\zeta \in \mathfrak{X}\left( M_{1}\times _{f}M_{2}\right) $
is a nontrivial $2-$Killing vector field. If $D_{\zeta }\zeta $ is parallel
along a curve $\gamma $, then%
\begin{equation*}
K\left( \zeta ,\dot{\gamma}\right) \geq 0.
\end{equation*}
\end{customthm}

Finally, in section 4, we apply these results on standard static
space-times and generalized Robertson-Walker space-times. For instance,
the following result is obtained.

\begin{customthm}{\ref{TH4}}
Let $\bar{M}=I_{f}\times M$ be a standard static space-time with the metric $%
\bar{g}=-f^{2}dt^{2}\oplus g$. Suppose that $u:I\rightarrow \mathbb{R}$ is
smooth and $\zeta $ is a vector field on $F$. Then $\bar{\zeta}=u\partial
_{t}+\zeta $ is a $2-$Killing vector field on $\bar{M}$ if one of the the
following conditions is satisfied:

\begin{enumerate}
\item $\zeta $ is $2-$Killing on $M$, $u=a$ and $f\zeta \left( f\right) =b$
where $a,b\in \mathbb{R}.$

\item $\zeta $ is $2-$Killing on $M$, $u=\left( rt+s\right) ^{\frac{1}{3}}$ and $%
\zeta \left( f\right) =0$ where $r,s\in \mathbb{R}.$
\end{enumerate}
\end{customthm}

Furthermore, the converse of this result and many others on generalized
Robertson-Walker space-times are discussed.

\section{Preliminaries}

In this section, we will provide basic definitions and curvature formulas
about warped product manifolds and Killing vector fields.

Suppose that $\left(M_{1},g_{1},D_{1}\right)$ and $\left(M_{2},g_{2},D_{2}\right)$
are two $\mathcal C^{\infty }$ pseudo-Riemannian manifolds equipped with
Riemannian metrics $g_{i}$ where $D_{i}$ is the Levi-Civita connection of
the metric $g_{i}$ for $i=1,2.$ Further suppose that $\pi _{1}:M_{1}\times
M_{2}\rightarrow M_{1}$ and $\pi _{2}:M_{1}\times M_{2}\rightarrow M_{2}$ are
the natural projection maps of the Cartesian product $M_{1}\times M_{2}$
onto $M_{1}$ and $M_{2},$ respectively. If $f:M_{1}\rightarrow \left(
0,\infty \right) $ is a positive real-valued smooth function, then the warped
product manifold $M_{1}\times _{f}M_{2}$ is the the product manifold $%
M_{1}\times M_{2}$ equipped with the metric tensor $g=g_{1}\oplus g_{2}$
defined by%
\begin{equation*}
g=\pi _{1}^{\ast }\left( g_{1}\right) \oplus \left( f\circ \pi _{1}\right)
^{2}\pi _{2}^{\ast }\left( g_{2}\right)
\end{equation*}%
where $^{\ast }$ denotes the pull-back operator on tensors\cite%
{Bishop:1969,Oneill:1983}. The function $f$ is called the warping function
of the warped product manifold $M_{1}\times _{f}M_{2}$. In particular, if $%
f=1$, then $M_{1}\times _{1}M_{2}=M_{1}\times M_{2}$ is the usual Cartesian
product manifold. It is clear that the submanifold $M_{1}\times \{q\}$ is
isometric to $M_{1}$ for every $q\in M_{2}$. Moreover, $\{p\}\times M_{2}$ is
homothetic to $M_{2}$. Throughout this article we use the same notation for
a vector field and for its lift to the product manifold.

Let $D$ be the Levi-Civita connection of the metric tensor $g$. The
following proposition is well-known\cite{Bishop:1969}.

\begin{proposition}
\label{P1}Let $\left( M_{1}\times _{f}M_{2},g\right) $ be a Riemannian
warped product manifold with warping function $f>0$ on $M_{1}$. Then

\begin{enumerate}
\item $\displaystyle{D_{X_{1}}Y=D_{X_{1}}^{1}Y_{1}\in \mathfrak{X}\left( M_{1}\right) }$

\item $\displaystyle{D_{X_{1}}Y_{2}=D_{Y_{2}}X_{1}=\frac{X_{1}\left( f\right) }{f}Y_{2}}$

\item $\displaystyle{D_{X_{2}}Y_{2}=-fg_{2}\left( X_{2},Y_{2}\right) \nabla^1
f+D_{X_{2}}^{2}Y_{2}}$
\end{enumerate}

for all $X_{i},Y_{i}\in \mathfrak{X}\left( M_{i}\right) $, with $i=1,2$ where $%
\nabla^1 f$ is the gradient of $f$.
\end{proposition}

A vector field $\zeta \in \mathfrak{X}\left( M\right) $ on a pseudo-Riemannian
manifold $\left( M,g\right) $ with metric $g$ is called a Killing vector field if
\begin{equation*}
\mathcal{L}_{\zeta }g=0
\end{equation*}%
where $\mathcal{L}_{\zeta }$ is the Lie derivative on $M$ with respect to $%
\zeta $. One can redefine Killing vector fields using the following
identity. Let $\zeta $ be a vector field, then%
\begin{equation}
\left( \mathcal{L}_{\zeta }g\right) \left( X,Y\right) =g\left( D_{X}\zeta
,Y\right) +g\left( X,D_{Y}\zeta \right)
\end{equation}%
for any vector fields $X,Y\in \mathfrak{X}\left( M\right) $. A simple
yet useful characterization of Killing vector fields is given in the following
proposition. The proof is straightforward by using the symmetry in the above identity.

\begin{lemma}
If $\left( M,g,D\right) $ is a pseudo-Riemannian manifold with Riemannian
connection $D$. A vector field $\zeta \in \mathfrak{X}\left( M\right) $ is a
Killing vector field if and only if%
\begin{equation}
g\left( D_{X}\zeta ,X\right) =0
\end{equation}%
for any vector field $X\in \mathfrak{X}\left( M\right) $.
\end{lemma}

Now we consider Killing vector fields on Riemannian warped product
manifolds. The following simple result will help us to present a
characterization of Killing vector fields on warped product manifolds.

\begin{lemma}
Let $\zeta \in \mathfrak{X}\left( M_{1}\times _{f}M_{2}\right) $ be a vector
field on the pseudo-Riemannian warped product manifold $M_{1}\times _{f}M_{2}$ with
warping function $f$. Then for any vector field $X\in \mathfrak{X}\left(
M_{1}\times _{f}M_{2}\right) $ we have%
\begin{equation}
g\left( D_{X}\zeta ,X\right) =g_{1}\left( D_{X_{1}}^{1}\zeta
_{1},X_{1}\right) +f^{2}g_{2}\left( D_{X_{2}}^{2}\zeta _{2},X_{2}\right)
+f\zeta _{1}\left( f\right) \left\Vert X_{2}\right\Vert^{2}_2
\end{equation}
\end{lemma}

\begin{proof}
Using Proposition \ref{P1}, we get%
\begin{eqnarray*}
g\left( D_{X}\zeta ,X\right) &=&g_{1}\left( D_{X_{1}}^{1}\zeta
_{1}-fg_{2}\left( X_{2},\zeta _{2}\right) \nabla f,X_{1}\right)
+f^{2}g_{2}\left( D_{X_{2}}^{2}\zeta _{2}+\zeta _{1}\left( \ln f\right)
X_{2}\right. \\
&&\left. +X_{1}\left( \ln f\right) \zeta _{2},X_{2}\right) \\
&=&g_{1}\left( D_{X_{1}}^{1}\zeta _{1},X_{1}\right) -fg_{2}\left(
X_{2},\zeta _{2}\right) X_{1}\left( f\right) +f^{2}g_{2}\left(
D_{X_{2}}^{2}\zeta _{2},X_{2}\right) \\
&&+f\zeta _{1}\left( f\right) g_{2}\left( X_{2},X_{2}\right) +fX_{1}\left(
f\right) g_{2}\left( \zeta _{2},X_{2}\right) \\
&=&g_{1}\left( D_{X_{1}}^{1}\zeta _{1},X_{1}\right) +f^{2}g_{2}\left(
D_{X_{2}}^{2}\zeta _{2},X_{2}\right) +f\zeta _{1}\left( f\right) \left\Vert
X_{2}\right\Vert^{2}_2
\end{eqnarray*}
\end{proof}

The preceding two theorems give us a characterization of Killing vector fields on
warped product manifolds. They are immediate consequence of the previous
result.

\begin{theorem}
Let $\zeta =\left( \zeta _{1},\zeta _{2}\right) \in \mathfrak{X}\left(
M_{1}\times _{f}M_{2}\right) $ be a vector field on the pseudo-Riemannian warped
product manifold $M_{1}\times _{f}M_{2}$ with warping function $f$. Then $%
\zeta $ is a Killing vector field if one of the following conditions hold

\begin{enumerate}
\item $\zeta =\left( \zeta _{1},0\right) $ and $\zeta _{1}$ is a killing
vector field on $M_{1}$.

\item $\zeta =\left( 0,\zeta _{2}\right) $ and $\zeta _{2}$ is a killing
vector field on $M_{2}$.

\item $\zeta _{i}$ is a Killing vector field on $M_{i},$ for $i=1,2$ and $\zeta
_{1}\left( f\right) =0$.
\end{enumerate}
\end{theorem}

The converse of the above result is considered in the following result:

\begin{theorem}
Let $\zeta =\left( \zeta _{1},\zeta _{2}\right) \in \mathfrak{X}\left(
M_{1}\times _{f}M_{2}\right) $ be a killing vector field on the warped
product manifold $M_{1}\times _{f}M_{2}$ with warping function $f$. Then

\begin{enumerate}
\item $\zeta _{1}$ is a Killing vector field on $M_{1}$.

\item $\zeta _{2}$ is a Killing vector field on $M_{2}$ if $\zeta _{1}\left(
f\right) =0$.
\end{enumerate}
\end{theorem}

In \cite{Unal:2012}, the authors proved similar results on standard static
space-times using the following proposition.

\begin{proposition}
Let $\zeta =\left( \zeta _{1},\zeta _{2}\right) \in \mathfrak{X}\left(
M_{1}\times _{f}M_{2}\right) $ be a vector field on the warped product
manifold $M_{1}\times _{f}M_{2}$ with warping function $f$. Then%
\begin{equation}
\left( \mathcal{L}_{\zeta }g\right) \left( X,Y\right) =\left( \mathcal{L}%
_{\zeta _{1}}^{1}g_{1}\right) \left( X_{1},Y_{1}\right) +f^{2}\left(
\mathcal{L}_{\zeta _{2}}^{2}g_{2}\right) \left( X_{2},Y_{2}\right) +2f\zeta
_{1}\left( f\right) g_{2}\left( X_{2},Y_{2}\right)
\end{equation}%
where $\mathcal{L}_{\zeta _{i}}^{i}$ is the Lie derivative on $M_{i}$ with
respect to $\zeta _{i},$ for $i=1,2$.
\end{proposition}

\section{$2-$Killing vector fields}

In this section after we define and state fundamental results about
$2-$Killing vector fields, we obtain the main results of the paper.

A vector field $\zeta \in \mathfrak{X}\left( M\right) $ is called a $2-$%
Killing vector field on a pseudo-Riemannian manifold $\left( M,g\right) $ if%
\begin{equation}
\mathcal{L}_{\zeta }\mathcal{L}_{\zeta }g=0
\end{equation}%
where $\mathcal{L}_{\zeta }$ is the Lie derivative in the
direction of $\zeta$ on $M$\cite{Operea:2008}.

The following two results\cite{Operea:2008} are needed to exploit the above
definition.

\begin{proposition}
\label{P2}Let $\zeta \in \mathfrak{X}\left( M\right) $ be a vector field on
a pseudo-Riemannian manifold $M$. Then%
\begin{equation}
\left( \mathcal{L}_{\zeta }\mathcal{L}_{\zeta }g\right) \left( X,Y\right)
=g\left( D_{\zeta }D_{X}\zeta -D_{\left[ \zeta ,X\right] }\zeta ,Y\right)
+g\left( X,D_{\zeta }D_{Y}\zeta -D_{\left[ \zeta ,Y\right] }\zeta \right)
+2g\left( D_{X}\zeta ,D_{Y}\zeta \right)  \label{e4}
\end{equation}%
for any vector fields $X,Y\in \mathfrak{X}\left( M\right) $.
\end{proposition}

The following result is quite direct and helpful.

\begin{corollary}
\label{c1}A vector field $\zeta $ is $2-$Killing if and only if%
\begin{equation}
R\left( \zeta ,X,\zeta ,X\right) =g\left( D_{X}\zeta ,D_{X}\zeta \right)
+g\left( D_{X}D_{\zeta }\zeta ,X\right)
\end{equation}%
for any vector field $X\in \mathfrak{X}\left( M\right) .$
\end{corollary}

The symmetry of equation (\ref{e4}), shows that $\zeta $ is $2-$Killing if
and only if%
\begin{equation*}
g\left( D_{\zeta }D_{X}\zeta -D_{\left[ \zeta ,X\right] }\zeta ,X\right)
+g\left( D_{X}\zeta ,D_{X}\zeta \right) =0
\end{equation*}

\begin{example}
Let $M$ be the 2-dimensional Euclidean space, i.e,
$\left(\mathbb{R}^{2}, {\rm d}s^{2}\right) $ where $%
{\rm d}s^{2}={\rm d}x^{2}+{\rm d}y^{2}$.
A vector field $\zeta =u\partial _{x}+v\partial
_{y}\in \mathfrak{X}\left( M\right) $ is $2-$Killing if%
\begin{equation*}
\left( \mathcal{L}_{\zeta }^{I}\mathcal{L}_{\zeta }^{I}g_{I}\right) \left(
X,Y\right) = 0
\end{equation*}%
for any vector fields $X,Y$, where $\mathcal{L}_{\zeta }$ is the Lie
derivative on $\mathbb{R}^{2}$ with respect to $\zeta $. Now it is easy to
show that $\zeta $ is $2-$Killing vector field on $M$ if and only if
\begin{eqnarray*}
uu_{xx}+2u_{x}^{2} &=&0 \\
vv_{yy}+2v_{y}^{2} &=&0
\end{eqnarray*}
\end{example}

By making use of the above proposition one can get sufficient and
necessary conditions for a vector field $\zeta =\left( \zeta _{1},\zeta
_{2}\right) \in \mathfrak{X}\left( M_{1}\times _{f}M_{2}\right) $ to be a $2-$%
Killing on the pseudo-Riemannian warped product manifold $M_{1}\times _{f}M_{2}$.
The following theorem represents a similar such.

\begin{theorem}
\label{TH 1}Let $\zeta =\left( \zeta _{1},\zeta _{2}\right) \in \mathfrak{X}%
\left( M_{1}\times _{f}M_{2}\right) $ be a vector field on the warped
product manifold $M_{1}\times _{f}M_{2}$. Then%
\begin{eqnarray*}
\left( \mathcal{L}_{\zeta }\mathcal{L}_{\zeta }g\right) \left( X,Y\right)
&=&\left( \mathcal{L}_{\zeta _{1}}^{1}\mathcal{L}_{\zeta
_{1}}^{1}g_{1}\right) \left( X_{1},Y_{1}\right) +f^{2}\left( \mathcal{L}%
_{\zeta _{2}}^{2}\mathcal{L}_{\zeta _{2}}^{2}g_{2}\right) \left(
X_{2},Y_{2}\right) \\
&&+4f\zeta _{1}\left( f\right) \left( \mathcal{L}_{\zeta
_{2}}^{2}g_{2}\right) \left( X_{2},Y_{2}\right) +2f\zeta _{1}\left( \zeta
_{1}\left( f\right) \right) g_{2}\left( X_{2},Y_{2}\right) \\
&&+2\zeta _{1}\left( f\right) \zeta _{1}\left( f\right) g_{2}\left(
X_{2},Y_{2}\right)
\end{eqnarray*}%
for any vector fields $X,Y\in \mathfrak{X}\left( M_{1}\times
_{f}M_{2}\right) $.
\end{theorem}

\begin{proof}
See Appendix A
\end{proof}

The following results are direct consequences of the above theorem.

\begin{corollary} \label{Cor1}
Let $\zeta =\zeta _{1}+\zeta _{2}\in \mathfrak{X}\left( M_{1}\times
_{f}M_{2}\right) $ be a vector field on the warped product manifold of the
form $M_{1}\times _{f}M_{2}$. If $\zeta _{1}+\zeta _{2}$ is a $2-$Killing
vector field on $M_{1}\times _{f}M_{2}$, then $\zeta _{1}$ is a $2-$Killing
vector field on $M_{1}$.
\end{corollary}

\begin{corollary}
\label{c2}Let $\zeta \in \mathfrak{X}\left( M_{1}\times _{f}M_{2}\right) $
be a vector field on the warped product manifold of the form $M_{1}\times
_{f}M_{2}$. Suppose that $\zeta _{1}$ and $\zeta _{2}$ are $2-$Killing
vector fields on $M_{1}$ and $M_{2},$ respectively. Then $\zeta _{1}+\zeta
_{2}$ is a $2-$Killing vector field on $M_{1}\times _{f}M_{2}$ if and only if

\begin{enumerate}
\item $\zeta _{1}(f)=0$, or

\item $\zeta _{2}$ is a homothetic vector field on $M_{2}$ with homothetic
factor $c$ (i.e, $\mathfrak{L}_{\zeta _{2}}^{2}g_{2}=cg_{2}$) such that
\begin{equation*}
f\zeta _{1}(\zeta _{1}(f))+\zeta _{1}(f)\zeta _{1}(f)=-2cf\zeta _{1}(f)
\end{equation*}
\end{enumerate}
\end{corollary}

\begin{corollary} \label{Cor3}
Let $\zeta =\zeta _{1}+\zeta _{2}\in \mathfrak{X}\left( M_{1}\times
_{f}M_{2}\right) $ be a vector field on the warped product manifold $%
M_{1}\times _{f}M_{2}$. Then $\zeta $ is a $2-$Killing vector field on $%
M_{1}\times _{f}M_{2}$ if one of the following conditions holds

\begin{enumerate}
\item the vector field $\zeta _{i}$ is a $2-$Killing vector field on $%
M_{i},i=1,2$, and $\zeta _{1}\left( f\right) =0$

\item $\zeta =\zeta _{2}$ and $\zeta _{2}$ is a $2-$Killing vector field on $%
M_{2}$.
\end{enumerate}
\end{corollary}

\begin{theorem}
\label{TH2}Let $\zeta \in \mathfrak{X}\left( M_{1}\times _{f}M_{2}\right) $
be a vector field on the warped product manifold $M_{1}\times _{f}M_{2}$.
Then

\begin{enumerate}
\item $\zeta =\zeta _{1}+\zeta _{2}$ is parallel if $\zeta _{i}$ is a $2-$%
Killing vector field, and $Ric^{i}\left( \zeta _{i},\zeta _{i}\right) \leq 0,$ $%
i=1,2$ and also $f$ is constant.

\item $\zeta =\zeta _{1}$ is parallel if $\zeta _{1}$ is a $2-$Killing
vector field, and $Ric^{1}\left( \zeta _{1},\zeta _{1}\right) \leq 0,$ and also $%
\zeta _{1}\left( f\right) =0$.

\item $\zeta =\zeta _{2}$ is parallel if $\zeta _{2}$ is a $2-$Killing
vector field, and $Ric^{2}\left( \zeta _{2},\zeta _{2}\right) \leq 0,$ and also $f$
is constant.
\end{enumerate}
\end{theorem}

\begin{proof}
Suppose that
\begin{equation*}
\left\{ e_{1},e_{2},...,e_{m}\right\}
\end{equation*}%
is an orthonormal frame in $T_{p}M_{1}$ and
\begin{equation*}
\left\{ e_{m+1},e_{m+2},...,e_{m+n}\right\}
\end{equation*}%
is an orthonormal frame in $T_{q}M_{2}$ for some point $\left( p,q\right)
\in M_{1}\times M_{2}$. Then
\begin{equation*}
\left\{ \overline{e}_{1},\overline{e}_{2},...,\overline{e}_{m+n}\right\}
\end{equation*}%
is an orthonormal frame in $T_{\left( p.q\right) }\left( M_{1}\times
M_{2}\right) $ where%
\begin{equation*}
\overline{e}_{i}=\left\{
\begin{array}{c}
e_{i} \\
\frac{1}{f}e_{i}%
\end{array}%
\right.
\begin{array}{c}
1\leq i\leq m \\
m+1\leq i\leq m+n%
\end{array}%
\end{equation*}%
Thus for any vector field $\zeta \in \mathfrak{X}\left( M_{1}\times
_{f}M_{2}\right) $ we have%
\begin{eqnarray}
Tr\left( g\left( D\zeta ,D\zeta \right) \right) &=&\sum_{i=1}^{m+n}g\left(
D_{\overline{e}_{i}}\zeta ,D_{\overline{e}_{i}}\zeta \right)  \notag \\
&=&\sum_{i=1}^{m}g\left( D_{e_{i}}\zeta ,D_{e_{i}}\zeta \right) +\frac{1}{%
f^{2}}\sum_{i=m+1}^{m+n}g\left( D_{e_{i}}\zeta ,D_{e_{i}}\zeta \right)
\label{e0}
\end{eqnarray}%
Using Proposition \ref{P1}, the first term is given by%
\begin{eqnarray}
\sum_{i=1}^{m}g\left( D_{e_{i}}\zeta ,D_{e_{i}}\zeta \right)
&=&\sum_{i=1}^{m}g\left( D_{e_{i}}^{1}\zeta _{1}+e_{i}\left( \ln f\right)
\zeta _{2},D_{e_{i}}^{1}\zeta _{1}+e_{i}\left( \ln f\right) \zeta _{2}\right)
\notag \\
&=&\sum_{i=1}^{m}g\left( D_{e_{i}}^{1}\zeta _{1},D_{e_{i}}^{1}\zeta
_{1}\right) +\sum_{i=1}^{m}g\left( e_{i}\left( \ln f\right) \zeta
_{2},e_{i}\left( \ln f\right) \zeta _{2}\right)  \notag \\
&=&Tr\left( g_{1}\left( D^{1}\zeta _{1},D^{1}\zeta _{1}\right) \right)
+\left\Vert \zeta _{2}\right\Vert ^{2}_2 \sum_{i=1}^{m}\left( e_{i}\left( \ln
f\right) \right) ^{2}  \notag \\
&=&Tr\left( g_{1}\left( D^{1}\zeta _{1},D^{1}\zeta _{1}\right) \right)
+\left\Vert \zeta _{2}\right\Vert ^{2}_2 \left\Vert \nabla f\right\Vert ^{2}_1
\label{e1}
\end{eqnarray}%
and the second term is given by%
\begin{eqnarray}
&&\frac{1}{f^{2}}\sum_{i=m+1}^{m+n}g\left( D_{e_{i}}\zeta ,D_{e_{i}}\zeta
\right)  \notag \\
&=&\frac{1}{f^{2}}\sum_{i=m+1}^{m+n}g\left( \zeta _{1}\left( \ln f\right)
e_{i}+D_{e_{i}}^{2}\zeta _{2}-fg_{2}\left( e_{i},\zeta _{2}\right) \nabla
f,\zeta _{1}\left( \ln f\right) e_{i}\right.  \notag \\
&&\left. +D_{e_{i}}^{2}\zeta _{2}-fg_{2}\left( e_{i},\zeta _{2}\right)
\nabla f\right) \\
&=&n\left( \zeta _{1}\left( \ln f\right) \right)
^{2}+\sum_{i=m+1}^{m+n}g_{2}\left( D_{e_{i}}^{2}\zeta
_{2},D_{e_{i}}^{2}\zeta _{2}\right) +\left\Vert \nabla f\right\Vert
^{2}_1 \sum_{i=m+1}^{m+n}\left( g_{2}\left( e_{i},\zeta _{2}\right) \right) ^{2}
\notag
\end{eqnarray}%
\begin{eqnarray}
&&\frac{1}{f^{2}}\sum_{i=m+1}^{m+n}g\left( D_{e_{i}}\zeta ,D_{e_{i}}\zeta
\right)  \notag \\
&=&\frac{n}{f^{2}}\left( \zeta _{1}\left( f\right) \right) ^{2}+Tr\left(
g_{2}\left( D^{2}\zeta _{2},D^{2}\zeta _{2}\right) \right) +\left\Vert
\nabla f\right\Vert ^{2}_1 \left\Vert \zeta _{2}\right\Vert ^{2}_2  \label{e2}
\end{eqnarray}

By using Equations (\ref{e1}) and (\ref{e2}), Equation (\ref{e0}) becomes%
\begin{eqnarray}
&&Tr\left( g\left( D\zeta ,D\zeta \right) \right)  \notag \\
&=&Tr\left( g_{1}\left( D^{1}\zeta _{1},D^{1}\zeta _{1}\right) \right)
+Tr\left( g_{2}\left( D^{2}\zeta _{2},D^{2}\zeta _{2}\right) \right)
+2\left\Vert \zeta _{2}\right\Vert ^{2}_2 \left\Vert \nabla f\right\Vert ^{2}_1 \\
&&+\frac{n}{f^{2}}\left( \zeta _{1}\left( f\right) \right) ^{2}  \label{e3}
\end{eqnarray}%
Now suppose that $\zeta _{i}$ is a $2-$Killing vector field and $%
Ric^{i}\left( \zeta _{i},\zeta _{i}\right) \leq 0,$ then $\zeta _{i}$ is a
parallel vector field with respect to the metric $g_{i}$ and hence%
\begin{equation*}
Tr\left( g_{1}\left( D^{1}\zeta _{1},D^{1}\zeta _{1}\right) \right)
=Tr\left( g_{2}\left( D^{2}\zeta _{2},D^{2}\zeta _{2}\right) \right) =0
\end{equation*}%
Then for a constant function $f,$ we have%
\begin{equation*}
Tr\left( g\left( D\zeta ,D\zeta \right) \right) =0
\end{equation*}%
Thus $\zeta $ is a parallel vector field with respect to the metric $g$. One
easily can prove the last two assertions using Equation \ref{e3}.
\end{proof}

\begin{corollary}
Let $\zeta \in \mathfrak{X}\left( M_{1}\times _{f}M_{2}\right) $ be a vector
field on a warped product manifold of the form $M_{1}\times _{f}M_{2}$. Then%
\begin{eqnarray*}
&&Tr\left( g\left( D\zeta ,D\zeta \right) \right) \\
&=&Tr\left( g_{1}\left( D^{1}\zeta _{1},D^{1}\zeta _{1}\right) \right)
+Tr\left( g_{2}\left( D^{2}\zeta _{2},D^{2}\zeta _{2}\right) \right)
+2\left\Vert \zeta _{2}\right\Vert ^{2}_2 \left\Vert \nabla f\right\Vert ^{2}_1 \\
&&+\frac{n}{f^{2}}\zeta _{1}\left( f\right) \zeta _{1}\left( f\right)
\end{eqnarray*}
\end{corollary}

\begin{theorem}
\label{TH3} Assume that $\zeta \in \mathfrak{X}\left( M_{1}\times _{f}M_{2}\right) $
is a non-trivial $2-$Killing vector field on the warped product manifold
$M_{1}\times _{f}M_{2}$. If $D_{\zeta }\zeta $ is parallel
along a curve $\gamma $, then%
\begin{equation*}
K\left( \zeta ,\dot{\gamma}\right) \geq 0.
\end{equation*}
\end{theorem}

\begin{proof}
Let $\zeta \in \mathfrak{X}\left( M_{1}\times _{f}M_{2}\right) $ be a
non-trivial $2-$Killing vector field, then%
\begin{eqnarray*}
0 &=&g\left( D_{\zeta }D_{X}\zeta ,Y\right) -g\left( D_{\left[ \zeta ,X%
\right] }\zeta ,Y\right) +2g\left( D_{X}\zeta ,D_{Y}\zeta \right) \\
&&+g\left( X,D_{\zeta }D_{Y}\zeta \right) -g\left( X,D_{\left[ \zeta ,Y%
\right] }\zeta \right)
\end{eqnarray*}%
for any vector fields $X,Y\in \mathfrak{X}\left( M_{1}\times
_{f}M_{2}\right) $. Take $X=Y=T=\dot{\gamma}$, then%
\begin{eqnarray*}
g\left( D_{\zeta }D_{T}\zeta ,T\right) -g\left( D_{\left[ \zeta ,T\right]
}\zeta ,T\right) +g\left( D_{T}\zeta ,D_{T}\zeta \right) &=&0 \\
g\left( D_{\zeta }D_{T}\zeta -D_{\left[ \zeta ,T\right] }\zeta ,T\right)
&=&-g\left( D_{T}\zeta ,D_{T}\zeta \right)
\end{eqnarray*}%
Since $D_{\zeta }\zeta $ is parallel along a curve $\gamma $, $D_{T}D_{\zeta
}\zeta =0$ and hence%
\begin{eqnarray*}
g\left( R\left( \zeta ,T\right) \zeta ,T\right) &=&-g\left( D_{T}\zeta
,D_{T}\zeta \right) \\
R\left( \zeta ,T,T,\zeta \right) &=&-g\left( D_{T}\zeta ,D_{T}\zeta \right)
\\
K\left( \zeta ,\dot{\gamma}\right) &=&\left\Vert D_{T}\zeta \right\Vert
^{2}\ast A\left( \zeta ,\dot{\gamma}\right) \geq 0
\end{eqnarray*}%
where $A\left( \zeta ,\dot{\gamma}\right) $ is area of the parallelogram
generated by $\zeta $ and $\dot{\gamma}$.
\end{proof}

The above result can be proved by using Corollary \ref{c1} as follows:

Let $\zeta \in \mathfrak{X}\left( M_{1}\times _{f}M_{2}\right) $ be a non-trivial
$2-$Killing vector field, then%
\begin{eqnarray*}
R\left( \zeta ,T,\zeta ,T\right) &=&g\left( D_{T}\zeta ,D_{T}\zeta \right)
+g\left( D_{T}D_{\zeta }\zeta ,T\right) \\
&=&\left\Vert D_{T}\zeta \right\Vert ^{2}+0 \\
&=&\left\Vert D_{T}\zeta \right\Vert ^{2}\geq 0
\end{eqnarray*}%

Moreover, if $D_{\zeta }\zeta =0$, then $K\left( \zeta ,X\right) \geq 0$ for
any vector field $X\in \mathfrak{X}\left( M_{1}\times _{f}M_{2}\right) $.

Now, we will state yet another condition for a vector field on warped
product manifolds to be $2-$Killing.

Let $(M,g)$ be an $n$-dimensional pseudo-Riemannian manifold. Suppose that $%
X $ and $Y$ are vector fields on $M.$ Then denote:

\begin{equation*}
\mathfrak{F}(X,Y)=g(\nabla _{X}\nabla _{Y}X,Y)+g(\nabla _{Y}X,\nabla
_{Y}X)-g(\nabla _{\lbrack X,Y]}X,Y)
\end{equation*}

Note that $X$ is a $2-$Killing vector field if $\mathfrak{F}(X,Y)=0$ for any
vector field $Y$ on $M$. We can prove many of the above results using the
following theorem.

\begin{theorem}
Let $\zeta \in \mathfrak{X}\left( M_{1}\times _{f}M_{2}\right) $ be a vector
field on the warped product manifold of the form $M_{1}\times _{f}M_{2}$.
Then
\begin{eqnarray*}
\mathfrak{F}(\zeta _{1}+\zeta _{2},X_{1}+X_{2}) &=&\mathfrak{F}_{1}(\zeta
_{1},X_{1})+f^{2}\mathfrak{F}_{2}(\zeta _{2},X_{2}) \\
&&+\left( f\zeta _{1}(f)+\zeta _{1}(f)\zeta _{1}(f)\right) g_{2}(X_{2},X_{2})
\\
&&+2f\zeta _{1}(f)g_{2}(\nabla _{X_{2}}\zeta _{2},X_{2}))
\end{eqnarray*}
\end{theorem}

\section{$2-$Killing Vector fields of Warped Product Space-Times}

We will apply our main results to some well-known warped product space-time
models to characterize their $2-$Killing vector fields.

\subsection{$2-$Killing Vector fields of Standard Static Space-Times}

We begin by defining standard static space-times:

Let $(M,g)$ be an $n-$dimensional Riemannian manifold and $f:M\rightarrow
(0,\infty )$ be a smooth function. Then $(n+1)-$dimensional product manifold
$I \times M$ furnished with the metric tensor
\begin{equation*}
\bar{g}=-f^{2}{\rm d}t^{2}\oplus g
\end{equation*}%
is called a standard static space-time and is denoted by $\bar{M}%
=I_{f}\times M$ where $I$ is an open, connected subinterval of $\mathbb R$
and ${\rm d}t^{2}$ is the Euclidean metric tensor on $I$.

Note that standard static space-times can be considered as a generalization
of the Einstein static universe%
\cite{AD1,AD,GES,Besse2008,Unal:2012,Unal:2009,Unal:2011,Unal:2004}.

\begin{theorem}
\label{TH4}Let $\bar{M}=I_{f}\times M$ be a standard static space-time with
the metric $\bar{g}=-f^{2}dt^{2}\oplus g$. Suppose that $u:I\rightarrow
\mathbb{R}$ is smooth on $I$. Then $\bar{\zeta}=u\partial _{t}+\zeta $ with
$\zeta \in \mathfrak{X}\left( M\right) $ is a $2-$Killing vector field on $%
\bar{M}$ if one of the following conditions is satisfied:

\begin{enumerate}
\item $\zeta $ is $2-$Killing on $M$, $u=a$ and $f\zeta \left( f\right) =b$
where $a,b\in \mathbb{R}.$

\item $\zeta $ is $2-$Killing on $M$, $u=\left( rt+s\right) ^{\frac{1}{3}}$ and $%
\zeta \left( f\right) =0$ where $r,s\in\mathbb{R}.$
\end{enumerate}
\end{theorem}

\begin{proof}
Let $\bar{X}=x\partial _{t}+X\in \mathfrak{X}\left( \bar{M}\right) $ and $%
\bar{Y}=y\partial _{t}+Y\in \mathfrak{X}\left( \bar{M}\right) $ be any
vector fields on $\bar{M}$ where $X,Y\in \mathfrak{X}\left( M\right) $ and $%
x,y$ are smooth real-valued functions on $I$. Using Theorem \ref{TH 1}, we
have%
\begin{eqnarray*}
&&\left( \mathcal{\bar{L}}_{\bar{\zeta}}\mathcal{\bar{L}}_{\bar{\zeta}}\bar{g%
}\right) \left( \bar{X},\bar{Y}\right) \\
&=&\left( \mathcal{L}_{\zeta }\mathcal{L}_{\zeta }g\right) \left( X,Y\right)
+f^{2}\left( \mathcal{L}_{u\partial _{t}}^{I}\mathcal{L}_{u\partial
_{t}}^{I}g_{I}\right) \left( x\partial _{t},y\partial _{t}\right) +4f\zeta
\left( f\right) \left( \mathcal{L}_{\zeta _{2}}^{2}g_{2}\right) \left(
x\partial _{t},y\partial _{t}\right) \\
&&+2f\zeta \left( \zeta \left( f\right) \right) g_{I}\left( x\partial
_{t},y\partial _{t}\right) +2\zeta \left( f\right) \zeta \left( f\right)
g_{I}\left( x\partial _{t},y\partial _{t}\right)
\end{eqnarray*}

Note that for a vector $u\partial _{t}$ field on $I$, we have%
\begin{eqnarray*}
\mathcal{L}_{\zeta }g_{I}\left( x\partial _{t},y\partial _{t}\right) &=&2%
\dot{u}g_{I}\left( x\partial _{t},y\partial _{t}\right) \\
\mathcal{L}_{\zeta }\mathcal{L}_{\zeta }g_{I}\left( x\partial _{t},y\partial
_{t}\right) &=&\mathfrak{(}2u\ddot{u}+4\dot{u}^{2}\mathfrak{)}g_{I}\left(
x\partial _{t},y\partial _{t}\right)
\end{eqnarray*}%
Then%
\begin{eqnarray}
&&\left( \mathcal{\bar{L}}_{\bar{\zeta}}\mathcal{\bar{L}}_{\bar{\zeta}}\bar{g%
}\right) \left( \bar{X},\bar{Y}\right)  \notag \\
&=&\left( \mathcal{L}_{\zeta }\mathcal{L}_{\zeta }g\right) \left( X,Y\right)
+f^{2}\mathfrak{(}2u\ddot{u}+4\dot{u}^{2}\mathfrak{)}g_{I}\left( x\partial
_{t},y\partial _{t}\right) +8\dot{u}f\zeta \left( f\right) g_{I}\left(
x\partial _{t},y\partial _{t}\right)  \notag \\
&&+2\zeta \left( f\zeta \left( f\right) \right) g_{I}\left( x\partial
_{t},y\partial _{t}\right)  \label{e5}
\end{eqnarray}

The vector field $\zeta $ is $2-$Killing on $M$ and the function $u$ in both
Condition (1) and Condition (2) is a solution of%
\begin{equation*}
\mathfrak{(}2u\ddot{u}+4\dot{u}^{2}\mathfrak{)=}0
\end{equation*}%
Thus equation (\ref{e5}) becomes%
\begin{equation}
\left( \mathcal{\bar{L}}_{\bar{\zeta}}\mathcal{\bar{L}}_{\bar{\zeta}}\bar{g}%
\right) \left( \bar{X},\bar{Y}\right) =2\left[ 4f\zeta \left( f\right) \dot{u%
}+\zeta \left( f\zeta \left( f\right) \right) \right] g_{I}\left( x\partial
_{t},y\partial _{t}\right)  \label{e6}
\end{equation}

Finally, Condition (1) implies that $\dot{u}=\zeta \left( f\zeta \left(
f\right) \right) =0$ and Condition (2) implies that $\zeta \left( f\right) =0$%
. \ Consequently, Condition (1) or Condition (2) implies that%
\begin{equation*}
\left( \mathcal{\bar{L}}_{\bar{\zeta}}\mathcal{\bar{L}}_{\bar{\zeta}}\bar{g}%
\right) \left( \bar{X},\bar{Y}\right) =0
\end{equation*}%
and so $\bar{\zeta}$ is $2-$Killing on $\bar{M}$.
\end{proof}

The converse of the above theorem is considered in the following corollary.
The proof is straightforward.

\begin{corollary}
Assume that $\bar{M}$ is a standard static space-time of the form $I_{f}\times M$
and $\bar{\zeta}=u\partial _{t}+\zeta $ is a $2-$%
Killing vector field on $\bar{M}$. Then $\zeta $ is a $2-$Killing vector
field on $M$. Moreover, the vector field $u\partial _{t}$ is a $2-$Killing
vector field on $I$ if $\zeta \left( f\right) =0$.
\end{corollary}

\begin{example}
Let $\zeta =u\left( t\right) \partial _{t}+v\left( x\right) \partial _{x}$\
be a vector field on the warped product manifold $\bar{M}=I_{f}\times
\mathbb{R}$\ with warping function $f$\ and the metric tensor $%
{\rm d}s^{2}=-f^{2}{\rm d}t^{2}+{\rm d}x^{2}$. To prove that $\zeta $ is a $2-$Killing vector
field, we can use Equation (\ref{e5}). If $\bar{X}=x\partial _{t}+X$ and $%
\bar{Y}=y\partial _{t}+Y$ are two vector fields on $\bar{M},$then
\begin{eqnarray}
\left( \mathcal{\bar{L}}_{\bar{\zeta}}\mathcal{\bar{L}}_{\bar{\zeta}}\bar{g}%
\right) \left( \bar{X},\bar{Y}\right) &=&\left( \mathcal{L}_{\zeta }\mathcal{%
L}_{\zeta }g\right) \left( X,Y\right) +f^{2}\mathfrak{(}2u\ddot{u}+4\dot{u}%
^{2}\mathfrak{)}g_{I}\left( x\partial _{t},y\partial _{t}\right)  \label{e7}
\\
&&+8\dot{u}f\zeta \left( f\right) g_{I}\left( x\partial _{t},y\partial
_{t}\right) +2\zeta \left( f\zeta \left( f\right) \right) g_{I}\left(
x\partial _{t},y\partial _{t}\right) ,  \notag
\end{eqnarray}%
where $\zeta =v\left( x\right) \partial _{x}$ and $g={\rm d}x^{2}$. It is
now easy to show that%
\begin{eqnarray*}
\zeta \left( f\right) &=&vf^{\prime }\text{,\ \ \ }\zeta \left( f\zeta
\left( f\right) \right) =v^{2}ff^{\prime \prime }+v^{2}f^{\prime
2}+vv^{\prime }ff^{\prime } \\
\left( \mathcal{L}_{\zeta }\mathcal{L}_{\zeta }g\right) \left( \partial
_{x},\partial _{x}\right) &=&2vv^{\prime \prime }+4v^{\prime 2}
\end{eqnarray*}%
Also, an orthogonal basis of $\mathfrak{X}\left( M\right) $ is $\left\{
\partial _{t},\partial _{x}\right\} .$ Thus Equation (\ref{e7}) becomes%
\begin{eqnarray*}
\left( \mathcal{\bar{L}}_{\bar{\zeta}}\mathcal{\bar{L}}_{\bar{\zeta}}\bar{g}%
\right) \left( \partial _{x},\partial _{x}\right) &=&2vv^{\prime \prime
}+4v^{\prime 2} \\
\left( \mathcal{\bar{L}}_{\bar{\zeta}}\mathcal{\bar{L}}_{\bar{\zeta}}\bar{g}%
\right) \left( \partial _{x},\partial _{t}\right) &=&0 \\
\left( \mathcal{\bar{L}}_{\bar{\zeta}}\mathcal{\bar{L}}_{\bar{\zeta}}\bar{g}%
\right) \left( \partial _{t},\partial _{x}\right) &=&0 \\
\left( \mathcal{\bar{L}}_{\bar{\zeta}}\mathcal{\bar{L}}_{\bar{\zeta}}\bar{g}%
\right) \left( \partial _{t},\partial _{t}\right) &=&-f^{2}\mathfrak{(}2u%
\ddot{u}+4\dot{u}^{2}\mathfrak{)}-8\dot{u}vff^{\prime }-2v^{2}ff^{\prime
\prime }-2v^{2}f^{\prime 2}-2vv^{\prime }ff^{\prime }
\end{eqnarray*}

Now if $u\partial _{t}$ and $v\partial _{t}$ are $2-$Killing vector fields
on $I$ and $\mathbb{R},$ respectively, then%
\begin{equation*}
2u\ddot{u}+4\dot{u}^{2}=2vv^{\prime \prime }+4v^{\prime 2}=0
\end{equation*}%
Consequently, $\zeta $ is $2-$Killing if $f^{\prime }=0$. One can obtain the
same result by using the definition of $2-$Killing vector fields (see Appendix B).
\end{example}

\subsection{$2-$Killing Vector fields of Generalized Robertson-Walker Space-Times}

We first define generalized Robertson-Walker space-times:

Let $(M,g)$ be an $n-$dimensional Riemannian manifold and $f:I \rightarrow
(0,\infty )$ be a smooth function. Then $(n+1)-$dimensional product manifold
$I \times M$ furnished with the metric tensor
\begin{equation*}
\bar{g}=-{\rm d}t^{2}\oplus f^{2}g
\end{equation*}%
is called a generalized Robertson-Walker space-time and is denoted by $\bar{M}%
=I \times _{f} M$ where $I$ is an open, connected subinterval of $\mathbb R$
and ${\rm d}t^{2}$ is the Euclidean metric tensor on $I$.

This structure was introduced to the literature to extend Robertson-Walker
space-times \cite{Sanchez00, Sanchez98, Sanchez99}.

Due to {\it Corollary \ref{Cor1}}, we need to determine
$2-$Killing vector fields on $I.$ Suppose that $\zeta _{1}=h\partial _{t}$
is a vector field on $I$ where $h$ is a smooth function on $I$. Then%
\begin{eqnarray*}
(\mathfrak{L}_{h\partial _{t}}^{I}\mathfrak{L}_{h\partial
_{t}}^{I}g_{I})(\partial _{t},\partial _{t}) &=&-2hh^{\prime \prime
}-4(h^{\prime })^{2} \\
&=&-2\left( hh^{\prime \prime }+2(h^{\prime })^{2}\right)
\end{eqnarray*}%
Therefore, $\zeta _{1}=h\partial _{t}$ is a $2-$Killing vector field on $I$
if and only if $hh^{\prime \prime }=-2(h^{\prime })^{2}$.

In this case, one
can solve the last differential equation and obtain that: $h(t)=\left(
at-b\right) ^{\frac{1}{3}}$ for some $a,b\in \mathbb{R}$ where $t\in I$ and $%
t\neq \frac{b}{a}$.

Thus to characterize $2-$Killing vector fields on the generalized
Robertson-Walker space-time of the form $\bar{M}=I\times _{f}M$,
one can focus on vector fields of the form $%
\left( at-b\right) ^{\frac{1}{3}}\partial _{t}+V$.

An easy application of \textit{Corollary \ref{c2}} leads us to the
following result.

\begin{proposition}
Let $\bar{M}=I\times _{f}M$ be a generalized Robertson-Walker space-time
with the metric tensor $\bar{g}=-dt^{2}\oplus f^{2}g$. Suppose that $V$ is a
$2-$Killing vector field on $(M,g).$ Then a vector field $\left( at-b\right)
^{\frac{1}{3}}\partial _{t}+V$ is a $2-$Killing vector field on $(\bar{M},%
\bar{g})$ if $V$ is a homothetic vector field on $(M,g)$ with $c$ satisfying%
\begin{equation*}
\frac{a}{3}f\dot{f}+\left( f\ddot{f}+\dot{f}^{2}\right) \left( at-b\right)
=-2cf\dot{f}\left( at-b\right) ^{\frac{2}{3}}
\end{equation*}
\end{proposition}

\begin{remark}
At this point, we want to emphasize that we prefer not to apply
{\it Corollary \ref{Cor3}} since Condition (1) implies that
the warping function $f$ of a generalized Robertson-Walker space-time
of the form $\bar{M}=I\times _{f}M$ is constant and hence the underlying
warped product turns out to be just a trivial product.
\end{remark}

\appendix

\section{Proof of Theorem \protect\ref{TH 1}}

Using Proposition \ref{P1} and Proposition \ref{P2}, we get%
\begin{eqnarray*}
\left( \mathcal{L}_{\zeta }\mathcal{L}_{\zeta }g\right) \left( X,Y\right)
&=&g\left( D_{\zeta }D_{X}\zeta ,Y\right) +g\left( X,D_{\zeta }D_{Y}\zeta
\right) -g\left( D_{\left[ \zeta ,X\right] }\zeta ,Y\right) -g\left( X,D_{%
\left[ \zeta ,Y\right] }\zeta \right)  \\
&&+2g\left( D_{X}\zeta ,D_{Y}\zeta \right)
\end{eqnarray*}%
The first term $T_{1}$ is given by%
\begin{eqnarray*}
T_{1} &=&g\left( D_{\zeta }D_{X}\zeta ,Y\right)  \\
&=&g\left( D_{\zeta }\left( D_{X_{1}}^{1}\zeta _{1}+\frac{1}{f}\zeta
_{1}\left( f\right) X_{2}+\frac{1}{f}X_{1}\left( f\right) \zeta
_{2}+D_{X_{2}}^{2}\zeta _{2}-fg_{2}\left( X_{2},\zeta _{2}\right) \mathcal{%
\bigtriangledown }f\right) ,Y\right)  \\
&=&g\left( D_{\zeta _{1}}^{1}D_{X_{1}}^{1}\zeta _{1}+\frac{1}{f}\zeta
_{1}\left( \zeta _{1}\left( f\right) \right) X_{2}+\frac{1}{f}\zeta
_{1}\left( X_{1}\left( f\right) \right) \zeta _{2}+\frac{1}{f}\zeta
_{1}\left( f\right) D_{X_{2}}^{2}\zeta _{2}\right.  \\
&&\left. -\zeta _{1}\left( f\right) g_{2}\left( X_{2},\zeta _{2}\right)
\mathcal{\bigtriangledown }f-fg_{2}\left( X_{2},\zeta _{2}\right) D_{\zeta
_{1}}^{1}\mathcal{\bigtriangledown }f+\frac{1}{f}\left( D_{X_{1}}^{1}\zeta
_{1}\right) \left( f\right) \zeta _{2}\right.  \\
&&\left. +\frac{1}{f}\zeta _{1}\left( f\right) D_{\zeta _{2}}^{2}X_{2}-\zeta
_{1}\left( f\right) g_{2}\left( X_{2},\zeta _{2}\right) \mathcal{%
\bigtriangledown }f+\frac{1}{f}X_{1}\left( f\right) D_{\zeta _{2}}^{2}\zeta
_{2}\right.  \\
&&\left. -X_{1}\left( f\right) g_{2}\left( \zeta _{2},\zeta _{2}\right)
\mathcal{\bigtriangledown }f+D_{\zeta _{2}}^{2}D_{X_{2}}^{2}\zeta
_{2}-fg_{2}\left( D_{X_{2}}^{2}\zeta _{2},\zeta _{2}\right) \mathcal{%
\bigtriangledown }f\right.  \\
&&\left. -fg_{2}\left( D_{\zeta _{2}}^{2}X_{2},\zeta _{2}\right) \mathcal{%
\bigtriangledown }f-fg_{2}\left( X_{2},D_{\zeta _{2}}^{2}\zeta _{2}\right)
\mathcal{\bigtriangledown }f-g_{2}\left( X_{2},\zeta _{2}\right) \left(
\mathcal{\bigtriangledown }f\right) \left( f\right) \zeta _{2},Y\right)
\end{eqnarray*}%
and so%
\begin{eqnarray*}
T_{1} &=&g_{1}\left( D_{\zeta _{1}}^{1}D_{X_{1}}^{1}\zeta _{1},Y_{1}\right)
+f\zeta _{1}\left( \zeta _{1}\left( f\right) \right) g_{2}\left(
X_{2}.Y_{2}\right) +f\zeta _{1}\left( X_{1}\left( f\right) \right)
g_{2}\left( \zeta _{2},Y_{2}\right)  \\
&&+f\zeta _{1}\left( f\right) g_{2}\left( D_{X_{2}}^{2}\zeta
_{2},Y_{2}\right) -\zeta _{1}\left( f\right) Y_{1}\left( f\right)
g_{2}\left( X_{2},\zeta _{2}\right) -fg_{2}\left( X_{2},\zeta _{2}\right)
g_{1}\left( D_{\zeta _{1}}^{1}\mathcal{\bigtriangledown }f,Y_{1}\right)  \\
&&+f\left( D_{X_{1}}^{1}\zeta _{1}\right) \left( f\right) g_{2}\left( \zeta
_{2},Y_{2}\right) +f\zeta _{1}\left( f\right) g_{2}\left( D_{\zeta
_{2}}^{2}X_{2},Y_{2}\right) -\zeta _{1}\left( f\right) Y_{1}\left( f\right)
g_{2}\left( X_{2},\zeta _{2}\right)  \\
&&+fX_{1}\left( f\right) g_{2}\left( D_{\zeta _{2}}^{2}\zeta
_{2},Y_{2}\right) -X_{1}\left( f\right) Y_{1}\left( f\right) g_{2}\left(
\zeta _{2},\zeta _{2}\right) +f^{2}g_{2}\left( D_{\zeta
_{2}}^{2}D_{X_{2}}^{2}\zeta _{2},Y_{2}\right)  \\
&&-fY_{1}\left( f\right) g_{2}\left( D_{X_{2}}^{2}\zeta _{2},\zeta
_{2}\right) -fY_{1}\left( f\right) g_{2}\left( D_{\zeta _{2}}^{2}X_{2},\zeta
_{2}\right) -fY_{1}\left( f\right) g_{2}\left( X_{2},D_{\zeta _{2}}^{2}\zeta
_{2}\right)  \\
&&-f^{2}g_{2}\left( X_{2},\zeta _{2}\right) \left( \mathcal{\bigtriangledown
}f\right) \left( f\right) g_{2}\left( \zeta _{2},Y_{2}\right)  \\
&=&g_{1}\left( D_{\zeta _{1}}^{1}D_{X_{1}}^{1}\zeta _{1},Y_{1}\right)
+f^{2}g_{2}\left( D_{\zeta _{2}}^{2}D_{X_{2}}^{2}\zeta _{2},Y_{2}\right)  \\
&&+f\zeta _{1}\left( \zeta _{1}\left( f\right) \right) g_{2}\left(
X_{2},Y_{2}\right) +f\zeta _{1}\left( X_{1}\left( f\right) \right)
g_{2}\left( \zeta _{2},Y_{2}\right) +f\zeta _{1}\left( f\right) g_{2}\left(
D_{X_{2}}^{2}\zeta _{2},Y_{2}\right)  \\
&&-f\zeta _{1}\left( Y_{1}\left( f\right) \right) g_{2}\left( X_{2},\zeta
_{2}\right) +fg_{2}\left( X_{2},\zeta _{2}\right) \left( D_{\zeta
_{1}}^{1}Y_{1}\right) \left( f\right)  \\
&&+fg_{2}\left( \zeta _{2},Y_{2}\right) \left( D_{X_{1}}^{1}\zeta
_{1}\right) \left( f\right) +f\zeta _{1}\left( f\right) g_{2}\left( D_{\zeta
_{2}}^{2}X_{2},Y_{2}\right) -2\zeta _{1}\left( f\right) Y_{1}\left( f\right)
g_{2}\left( X_{2},\zeta _{2}\right)  \\
&&+fX_{1}\left( f\right) g_{2}\left( D_{\zeta _{2}}^{2}\zeta
_{2},Y_{2}\right) -X_{1}\left( f\right) Y_{1}\left( f\right) g_{2}\left(
\zeta _{2},\zeta _{2}\right)  \\
&&-fY_{1}\left( f\right) g_{2}\left( D_{X_{2}}^{2}\zeta _{2},\zeta
_{2}\right) -fY_{1}\left( f\right) g_{2}\left( D_{\zeta _{2}}^{2}X_{2},\zeta
_{2}\right) -fY_{1}\left( f\right) g_{2}\left( X_{2},D_{\zeta _{2}}^{2}\zeta
_{2}\right)  \\
&&-f^{2}g_{2}\left( X_{2},\zeta _{2}\right) g_{2}\left( \zeta
_{2},Y_{2}\right) \left( \mathcal{\bigtriangledown }f\right) \left( f\right)
\end{eqnarray*}%
Exchanging $X$ and $Y$ we get the second term $T_{2}$ and so%
\begin{eqnarray*}
&&T_{1}+T_{2} \\
&=&g\left( D_{\zeta }D_{X}\zeta ,Y\right) +g\left( D_{\zeta }D_{Y}\zeta
,X\right)  \\
&=&g_{1}\left( D_{\zeta _{1}}^{1}D_{X_{1}}^{1}\zeta _{1},Y_{1}\right)
+f^{2}g_{2}\left( D_{\zeta _{2}}^{2}D_{X_{2}}^{2}\zeta _{2},Y_{2}\right)
+g_{1}\left( D_{\zeta _{1}}^{1}D_{Y_{1}}^{1}\zeta _{1},X_{1}\right)  \\
&&+f^{2}g_{2}\left( D_{\zeta _{2}}^{2}D_{Y_{2}}^{2}\zeta _{2},X_{2}\right)
+2f\zeta _{1}\left( \zeta _{1}\left( f\right) \right) g_{2}\left(
X_{2},Y_{2}\right) -2X_{1}\left( f\right) Y_{1}\left( f\right) g_{2}\left(
\zeta _{2},\zeta _{2}\right)  \\
&&-2f^{2}g_{2}\left( X_{2},\zeta _{2}\right) g_{2}\left( \zeta
_{2},Y_{2}\right) \left( \mathcal{\bigtriangledown }f\right) \left( f\right)
+f\zeta _{1}\left( f\right) g_{2}\left( D_{X_{2}}^{2}\zeta _{2},Y_{2}\right)
\\
&&+fg_{2}\left( X_{2},\zeta _{2}\right) \left( D_{\zeta
_{1}}^{1}Y_{1}\right) \left( f\right) +fg_{2}\left( \zeta _{2},X_{2}\right)
\left( D_{Y_{1}}^{1}\zeta _{1}\right) \left( f\right)  \\
&&+f\zeta _{1}\left( f\right) g_{2}\left( D_{Y_{2}}^{2}\zeta
_{2},X_{2}\right) +fg_{2}\left( Y_{2},\zeta _{2}\right) \left( D_{\zeta
_{1}}^{1}X_{1}\right) \left( f\right) +fg_{2}\left( \zeta _{2},Y_{2}\right)
\left( D_{X_{1}}^{1}\zeta _{1}\right) \left( f\right)  \\
&&+f\zeta _{1}\left( f\right) g_{2}\left( D_{\zeta
_{2}}^{2}X_{2},Y_{2}\right) -2\zeta _{1}\left( f\right) Y_{1}\left( f\right)
g_{2}\left( X_{2},\zeta _{2}\right) -fY_{1}\left( f\right) g_{2}\left(
D_{X_{2}}^{2}\zeta _{2},\zeta _{2}\right)  \\
&&+f\zeta _{1}\left( f\right) g_{2}\left( D_{\zeta
_{2}}^{2}Y_{2},X_{2}\right) -2\zeta _{1}\left( f\right) X_{1}\left( f\right)
g_{2}\left( Y_{2},\zeta _{2}\right) -fX_{1}\left( f\right) g_{2}\left(
D_{Y_{2}}^{2}\zeta _{2},\zeta _{2}\right)  \\
&&-fY_{1}\left( f\right) g_{2}\left( D_{\zeta _{2}}^{2}X_{2},\zeta
_{2}\right) -fX_{1}\left( f\right) g_{2}\left( D_{\zeta _{2}}^{2}Y_{2},\zeta
_{2}\right)
\end{eqnarray*}%
The third term is given by%
\begin{eqnarray*}
&&T_{3} \\
&=&g\left( D_{\left[ \zeta ,X\right] }\zeta ,Y\right)  \\
&=&g\left( D_{\left[ \zeta _{1},X_{1}\right] }\zeta _{1}+D_{\left[ \zeta
_{2},X_{2}\right] }\zeta _{1}+D_{\left[ \zeta _{1},X_{1}\right] }\zeta
_{2}+D_{\left[ \zeta _{2},X_{2}\right] }\zeta _{2},Y\right)  \\
&=&g\left( D_{\left[ \zeta _{1},X_{1}\right] }^{1}\zeta _{1}+\frac{1}{f}%
\zeta _{1}\left( f\right) \left[ \zeta _{2},X_{2}\right] +\frac{1}{f}\left[
\zeta _{1},X_{1}\right] \left( f\right) \zeta _{2}+D_{\left[ \zeta _{2},X_{2}%
\right] }^{2}\zeta _{2}\right.  \\
&&\left. -\frac{{}}{{}}fg_{2}\left( \left[ \zeta _{2},X_{2}\right] ,\zeta
_{2}\right) \mathcal{\bigtriangledown }f,Y\right)  \\
&=&g_{1}\left( D_{\left[ \zeta _{1},X_{1}\right] }^{1}\zeta
_{1},Y_{1}\right) +f\zeta _{1}\left( f\right) g_{2}\left( \left[ \zeta
_{2},X_{2}\right] ,Y_{2}\right) +f\left[ \zeta _{1},X_{1}\right] \left(
f\right) g_{2}\left( \zeta _{2},Y_{2}\right)  \\
&&+f^{2}g_{2}\left( D_{\left[ \zeta _{2},X_{2}\right] }^{2}\zeta
_{2},Y_{2}\right) -fg_{2}\left( \left[ \zeta _{2},X_{2}\right] ,\zeta
_{2}\right) Y_{1}\left( f\right)  \\
&=&g_{1}\left( D_{\left[ \zeta _{1},X_{1}\right] }^{1}\zeta
_{1},Y_{1}\right) +f^{2}g_{2}\left( D_{\left[ \zeta _{2},X_{2}\right]
}^{2}\zeta _{2},Y_{2}\right) +f\zeta _{1}\left( f\right) g_{2}\left( \left[
\zeta _{2},X_{2}\right] ,Y_{2}\right)  \\
&&+fg_{2}\left( \zeta _{2},Y_{2}\right) \left[ \zeta _{1},X_{1}\right]
\left( f\right) -fg_{2}\left( \left[ \zeta _{2},X_{2}\right] ,\zeta
_{2}\right) Y_{1}\left( f\right)
\end{eqnarray*}%
Exchanging $X$ and $Y$ we get the fourth term $T_{4}$ and so%
\begin{eqnarray*}
&&T_{3}+T_{4} \\
&=&g_{1}\left( D_{\left[ \zeta _{1},X_{1}\right] }^{1}\zeta
_{1},Y_{1}\right) +g_{1}\left( D_{\left[ \zeta _{1},Y_{1}\right] }^{1}\zeta
_{1},X_{1}\right) +f^{2}g_{2}\left( D_{\left[ \zeta _{2},X_{2}\right]
}^{2}\zeta _{2},Y_{2}\right)  \\
&&+f^{2}g_{2}\left( D_{\left[ \zeta _{2},Y_{2}\right] }^{2}\zeta
_{2},X_{2}\right) +f\zeta _{1}\left( f\right) g_{2}\left( \left[ \zeta
_{2},X_{2}\right] ,Y_{2}\right) +fg_{2}\left( \zeta _{2},Y_{2}\right) \left[
\zeta _{1},X_{1}\right] \left( f\right)  \\
&&-fY_{1}\left( f\right) g_{2}\left( \left[ \zeta _{2},X_{2}\right] ,\zeta
_{2}\right) +f\zeta _{1}\left( f\right) g_{2}\left( \left[ \zeta _{2},Y_{2}%
\right] ,X_{2}\right) +fg_{2}\left( \zeta _{2},X_{2}\right) \left[ \zeta
_{1},Y_{1}\right] \left( f\right)  \\
&&-fX_{1}\left( f\right) g_{2}\left( \left[ \zeta _{2},Y_{2}\right] ,\zeta
_{2}\right)
\end{eqnarray*}%
The last term $T_{5}$ is given by%
\begin{eqnarray*}
&&(1/2)T_{5} \\
&=&g\left( D_{X}\zeta ,D_{Y}\zeta \right)  \\
&=&g\left( D_{X_{1}}^{1}\zeta _{1}+\frac{1}{f}\zeta _{1}\left( f\right)
X_{2}+\frac{1}{f}X_{1}\left( f\right) \zeta _{2}+D_{X_{2}}^{2}\zeta
_{2}-fg_{2}\left( X_{2},\zeta _{2}\right) \mathcal{\bigtriangledown }%
f,\right.  \\
&&\left. D_{Y_{1}}^{1}\zeta _{1}+\frac{1}{f}\zeta _{1}\left( f\right) Y_{2}+%
\frac{1}{f}Y_{1}\left( f\right) \zeta _{2}+D_{Y_{2}}^{2}\zeta
_{2}-fg_{2}\left( Y_{2},\zeta _{2}\right) \mathcal{\bigtriangledown }%
f\right)  \\
&=&g_{1}\left( D_{X_{1}}^{1}\zeta _{1},D_{Y_{1}}^{1}\zeta _{1}\right)
-fg_{2}\left( Y_{2},\zeta _{2}\right) \left( D_{X_{1}}^{1}\zeta _{1}\right)
\left( f\right) +\zeta _{1}\left( f\right) \zeta _{1}\left( f\right)
g_{2}\left( X_{2},Y_{2}\right)  \\
&&+\zeta _{1}\left( f\right) Y_{1}\left( f\right) g_{2}\left( X_{2},\zeta
_{2}\right) +f\zeta _{1}\left( f\right) g_{2}\left( X_{2},D_{Y_{2}}^{2}\zeta
_{2}\right)  \\
&&+\zeta _{1}\left( f\right) X_{1}\left( f\right) g_{2}\left( \zeta
_{2},Y_{2}\right) +X_{1}\left( f\right) Y_{1}\left( f\right) g_{2}\left(
\zeta _{2},\zeta _{2}\right) +fX_{1}\left( f\right) g_{2}\left( \zeta
_{2},D_{Y_{2}}^{2}\zeta _{2}\right)  \\
&&+f\zeta _{1}\left( f\right) g_{2}\left( D_{X_{2}}^{2}\zeta
_{2},Y_{2}\right) +fY_{1}\left( f\right) g_{2}\left( D_{X_{2}}^{2}\zeta
_{2},\zeta _{2}\right) +f^{2}g_{2}\left( D_{X_{2}}^{2}\zeta
_{2},D_{Y_{2}}^{2}\zeta _{2}\right)  \\
&&-fg_{2}\left( X_{2},\zeta _{2}\right) \left( D_{Y_{1}}^{1}\zeta
_{1}\right) \left( f\right) +f^{2}g_{2}\left( X_{2},\zeta _{2}\right)
g_{2}\left( Y_{2},\zeta _{2}\right) g_{1}\left( \mathcal{\bigtriangledown }f,%
\mathcal{\bigtriangledown }f\right)
\end{eqnarray*}%
Then%
\begin{eqnarray*}
\left( \mathcal{L}_{\zeta }\mathcal{L}_{\zeta }g\right) \left( X,Y\right)
&=&\left( \mathcal{L}_{\zeta _{1}}^{1}\mathcal{L}_{\zeta
_{1}}^{1}g_{1}\right) \left( X_{1},Y_{1}\right) +f^{2}\left( \mathcal{L}%
_{\zeta _{2}}^{2}\mathcal{L}_{\zeta _{2}}^{2}g_{2}\right) \left(
X_{2},Y_{2}\right)  \\
&&+4f\zeta _{1}\left( f\right) \left( \mathcal{L}_{\zeta
_{2}}^{2}g_{2}\right) \left( X_{2},Y_{2}\right) +2f\zeta _{1}\left( \zeta
_{1}\left( f\right) \right) g_{2}\left( X_{2},Y_{2}\right)  \\
&&+2\zeta _{1}\left( f\right) \zeta _{1}\left( f\right) g_{2}\left(
X_{2},Y_{2}\right)
\end{eqnarray*}

\section{Space-time Example}

In this section we deal with a standard static space-time
of the form $I_{f}\times
\mathbb{R}
$. Using Proposition \ref{P1}, one can establish the followings

\begin{enumerate}
\item $\nabla _{\partial _{x}}\partial _{x}=0$,

\item $\nabla _{\partial _{t}}\partial _{x}=\nabla _{\partial _{x}}\partial
_{t}=\partial _{x}\left( \ln f\right) \partial _{t}=\frac{f^{\prime }}{f}%
\partial _{t}$\ and

\item $\nabla _{\partial _{t}}\partial _{t}=ff^{\prime }\partial _{x}$
\end{enumerate}

on the warped product manifold $I_{f}\times
\mathbb{R}
$. It is clear that%
\begin{equation*}
\left[ \bar{\zeta},\partial _{t}\right] =-\dot{u}\partial _{t}\text{ \ \ \ \
\ \ \ \ }\left[ \bar{\zeta},\partial _{x}\right] =-v^{\prime }\partial _{x}
\end{equation*}%
Also, we have%
\begin{eqnarray*}
\nabla _{\partial _{t}}\bar{\zeta} &=&uff^{\prime }\partial _{x}+\frac{1}{f}%
\left( \dot{u}f+vf^{\prime }\right) \partial _{t} \\
\nabla _{\partial _{x}}\bar{\zeta} &=&v^{\prime }\partial _{x}+\frac{1}{f}%
\left( uf^{\prime }\right) \partial _{t}
\end{eqnarray*}

and%
\begin{eqnarray*}
\nabla _{\bar{\zeta}}\nabla _{\partial _{t}}\bar{\zeta} &=&\left[
uvff^{\prime \prime }+2uvf^{\prime 2}+2u\dot{u}ff^{\prime }\right] \partial
_{x} \\
&&+\frac{1}{f}\left[ v^{2}f^{\prime \prime }+vv^{\prime }f^{\prime }+v\dot{u}%
f^{\prime }-u^{2}ff^{\prime 2}+u\ddot{u}f\right] \partial _{t} \\
\nabla _{\bar{\zeta}}\nabla _{\partial _{x}}\bar{\zeta} &=&\left( vv^{\prime
\prime }+u^{2}f^{\prime 2}\right) \partial _{x}+\frac{1}{f}\left( u\dot{u}%
f^{\prime }+uv^{\prime }f^{\prime }+uvf^{\prime \prime }\right) \partial _{t}
\end{eqnarray*}%
Finally,%
\begin{eqnarray*}
\nabla _{\left[ \bar{\zeta},\partial _{t}\right] }\bar{\zeta} &=&-u\dot{u}%
ff^{\prime }\partial _{x}-\frac{1}{f}\left( \dot{u}vf^{\prime }+\dot{u}%
^{2}f\right) \partial _{t} \\
\nabla _{\left[ \bar{\zeta},\partial _{x}\right] }\bar{\zeta} &=&-v^{\prime
2}\partial _{x}-\frac{1}{f}\left( uv^{\prime }f^{\prime }\right) \partial
_{t}
\end{eqnarray*}

Now we can evaluate $2$-Killing forms  on $I_{f}\times
\mathbb{R}
$ as follows%
\begin{eqnarray*}
\left( \mathcal{\bar{L}}_{\bar{\zeta}}\mathcal{\bar{L}}_{\bar{\zeta}%
}g\right) \left( \partial _{x},\partial _{x}\right) &=&2\left[ vv^{\prime
\prime }+2v^{\prime 2}\right] \\
\left( \mathcal{L}_{\bar{\zeta}}\mathcal{L}_{\bar{\zeta}}g\right) \left(
\partial _{t},\partial _{x}\right) &=&0 \\
\left( \mathcal{L}_{\bar{\zeta}}\mathcal{L}_{\bar{\zeta}}g\right) \left(
\partial _{x},\partial _{t}\right) &=&0 \\
\left( \mathcal{L}_{\bar{\zeta}}\mathcal{L}_{\bar{\zeta}}g\right) \left(
\partial _{t},\partial _{t}\right) &=&-2f^{2}\left[ u\ddot{u}+2\dot{u}^{2}%
\right] -2\left[ v^{2}ff^{\prime \prime }+vv^{\prime }ff^{\prime }\right] -8%
\dot{u}vff^{\prime }-2v^{2}f^{\prime 2}
\end{eqnarray*}

which is what we have done before.

\bibliographystyle{acm}

\bigskip 

\end{document}